\documentclass[11pt]{amsart}

\usepackage{bbm}
\usepackage{amssymb}
\usepackage{amscd}
\usepackage[all]{xy}

\newtheorem{thm}{Theorem}[section]
\newtheorem{lem}[thm]{Lemma}
\newtheorem{cor}[thm]{Corollary}
\newtheorem{prop}[thm]{Proposition}

\theoremstyle{definition}

\newtheorem{defn}[thm]{Definition}

\newtheorem{ex}[thm]{Example}
\newtheorem{exs}[thm]{Examples}

\theoremstyle{remark}

\newtheorem{rem}[thm]{Remark}

\newcommand{\thmref}[1]{Theorem~\ref{#1}}
\newcommand{\corref}[1]{Corollary~\ref{#1}}
\newcommand{\secref}[1]{\S\ref{#1}}

\newcommand{\propref}[1]{Proposition~\ref{#1}}
\newcommand{\lemref}[1]{Lemma~\ref{#1}}

\newcommand{\exref}[1]{Example~\ref{#1}}

\newcommand{\Oo}{{\mathcal  O}}

\newcommand{\Sp}{{\mathcal Sp}}
\newcommand{\T}{{\mathcal T}}

\newcommand{\Ss}{{\mathbb  S}}

\newcommand{\Sinfty}{\Sigma^{\infty}}
\newcommand{\Oinfty}{\Omega^{\infty}}

\newcommand{\SnSp}{\Sigma_n\text{-}Sp}
\newcommand{\MCom}{\Mod\text{-}Com}
\newcommand{\ComM}{Com\text{-}\Mod}
\newcommand{\ComAlg}{Com\text{-}\Alg}

\newcommand{\Map}{\operatorname{Map}}

\newcommand{\Com}{\operatorname{Com}}

\newcommand{\sm}{\wedge}

\newcommand{\colim}{\operatorname*{colim}}

\newcommand{\Sym}{\mathsf{Sym}}
\newcommand{\co}{\circ_{\Oo}}
\newcommand{\ca}{\circ_{Com}}

\newcommand{\Alg}{\mathsf{Alg}}

\newcommand{\Mod}{\mathsf{Mod}}
\newcommand{\Epi}{\operatorname{Epi}}
\newcommand{\Mono}{\operatorname{Mono}}

\newcommand{\n}{{\bf n}}
\newcommand{\rr}{{\bf r}}
\newcommand{\m}{{\bf m}}

\newcommand{\ra}{\rightarrow}
\newcommand{\xra}{\xrightarrow}
\newcommand{\la}{\leftarrow}
\newcommand{\xla}{\xleftarrow}

\begin{document}

\title[Right Com-modules]{Applications of the circle product with a right Com--module to the theory of commutative ring spectra}

\author[Kuhn]{Nicholas J. Kuhn}
\address{Department of Mathematics, University of Virginia, Charlottesville, VA, USA}

%%%% Subject entries to be placed here %%%%
%%\subject{algebraic topology}

%%%% Keyword entries to be placed here %%%%

%%%% Insert corresponding author and its email address}
%%\corres{Nicholas Kuhn\\
\email{njk4x@virginia.edu}

\date{October 4, 2024}

%%%% Abstract text to be placed here %%%%%%%%%%%%
\begin{abstract}
If $Com$ is the reduced commutative operad, the category of $Com$-algebras in spectra  is the category of nounital commutative ring spectra.  The theme of this survey is that many important constructions on $Com$-algebras are given by taking the derived circle product with well chosen right $Com$-modules.

We show that examples of constructions arising this way include $K \otimes I$, the tensor product of a based space $K$ with such an algebra $I$, and $TQ(I)$, the Topological Andr\'e-Quillen homology spectrum of $I$.

We then show how filtrations of right Com-modules can be used to filter such constructions.

A natural decreasing filtration on right $Com$-modules is described. When specialized to the $Com$-bimodule $Com$, this defines the augmentation ideal tower of $I$, built out of the extended powers of $TQ(I)$.

A less studied natural increasing filtration is described. When specialized to the right Com-module used to define $TQ(I)$, one gets a filtration on $TQ(I)$ built out of $I$ and the spaces in the Lie cooperad. There are two versions of this in the literature, and our setting here makes it easy to prove that these agree. 

Much of this applies with $Com$ replaced by a more general reduced operad, and we make a few remarks about this.
\end{abstract}
%%%%%%%%%%%%%%%%%%%%%%%%%%%

%%%%%%%%%% Insert the texts which can accomdate on firstpage in the tag "fmtext" %%%%%

%%\begin{fmtext}
%%%  \section{Insert A head here}
%%%% Insert A head here
%%\end{fmtext}

%%%%%%%%%%%%%%% End of first page %%%%%%%%%%%%%%%%%%%%%

\maketitle

\section{Introduction} \label{intro sec}

If $Com$ is the reduced commutative operad in the symmetric monoidal category of spectra $(\Sp, \sm, \Ss)$, a $Com$--algebra $I$ is a nonunital commutative ring spectrum. These arise in many natural ways, e.g. as the Spanier-Whitehead duals of based spaces, or fibers of commutative ring spectra augmented over $\Ss$.

As we will review, a $Com$-algebra is a special kind of left $Com$-module. The theme of this survey is that many important constructions on $\ComAlg$, the category $Com$-algebras, arise by taking (derived) circle products $M\ca^LI$ with well chosen {\em right} $Com$-modules $M$.  This suggests looking for useful structure within $\MCom$, the category of right $Com$-modules.

We feature two constructions that can be written in this way: $TQ(I)$, the Topological Andr\'e-Quillen spectrum of an $Com$-algebra $I$ as in \cite{basterra, kuhn pereira}, and the tensor product $K \otimes I$, where $K$ is a based space and $I$ is a $Com$--algebra as in \cite{K tensor R}.  The right $Com$--module associated with $TQ(I)$ and the right $Com$-module associated to $\displaystyle \colim_k \Sigma^{-k} (S^k \otimes I)$ are cofibrant and weakly equivalent, and this leads to a simple and direct proof that 
$$TQ(I) \simeq \colim_k \Sigma^{-k} (S^k \otimes I),$$
so that the functor $I \mapsto TQ(I)$ can be viewed as stabilization in $\ComAlg$ \cite{basterra mccarthy}.

We feature two `dual' bits of structure on $\MCom$. 

We review that there is a canonical decreasing filtration on right $Com$-modules, as studied in \cite{kuhn pereira} and \cite{harper and hess}.  Specialization to $Com$, viewed as a right $Com$-module, leads to a tower converging to $I$ with $n$th fiber equivalent to $Com(n) \sm_{\Sigma_n} TQ(I)^{\sm n}$.  

We show that there is also a canonical increasing filtration on right $Com$-modules.  Applied to the appropriate right $Com$-modules, this leads to the canonical filtration of $K \otimes I$ of \cite{K tensor R}, the filtration of $TQ(I)$  also from \cite{K tensor R}, and the filtration of $TQ(I)$ defined in \cite{behrens rezk}.  Our approach makes it easy to answer in the affirmative an open question from \cite{behrens rezk}: the filtrations on $TQ(I)$ in \cite{K tensor R} and \cite{behrens rezk} are equivalent, as they arise from the canonical filtration of two right modules that are weakly equivalent and cofibrant.  The filtration on $TQ(I)$ has $n$th cofiber equivalent to $Lie(n) \sm_{\Sigma_n} I^{\sm n}$, where $Lie$ is the spectral cooperad $BCom$ of \cite{ching}.

We describe the organization of the paper.

In \secref{background sec}, we will review some basic constructions related to operads and their modules. In particular, we remind the reader that operads and their modules are special sorts of symmetric sequences in $\Sp$: sequences $M = \{M(n) \ | \ n = 0,1,2,\dots\}$ such that $M(n)$ is a spectrum with an action of the $n$th symmetric group $\Sigma_n$. The category of these, $\Sym(Sp)$, is a monoidal category using the circle product $\circ$, with  unit $\Ss(1) = \{*,\Ss,*,*, \dots\}$.  The reduced commutative operad is $Com = \{*,\Ss,\Ss,\Ss, \dots\}$, and  we will consider only {\em reduced} right $Com$--modules: right modules $M$ with $M(0) = *$.   Also in this section we recall some homotopical properties of these constructions when working in the model category of symmetric spectra used in \cite{kuhn pereira},  and discuss the induced model structure on right $\Com$--modules as defined in \cite{fresse}.

In \secref{bar construction and TQ section} we review the definition and basic properties of the bar construction $B(M,Com,N)$, where $M$ and $N$ are right and left $\Com$-modules. In particular, $B(M,Com,Com)$ is a cofibrant replacement for $M$, and thus $B(M,Com,I^c)$ is a model for $M \ca^L I$, where $I^c$ is a cofibrant replacements for $I$.  An important special case is the `spectrum of derived indecomposables of $I$', defined by $TQ(I) = B(\Ss(1),Com,I^c)$; this definition goes back to \cite{basterra}.

In \secref{tensor sec}, we define a $Com$-bimodule $(K \otimes Com)$ by the formula $(K \otimes Com)(n) = K^{\sm n} \sm \Ss$, and show that $(K \otimes Com)\circ_{Com} I$ is isomorphic to  $K \otimes I$. This is a variant of the author's old observations in \cite{K tensor R}.  We note that $(K \otimes Com)$ will be cofibrant as a right $Com$-module.

In \secref{taq is stablization sec},  we observe that both $B(\Ss(1), Com, Com)$ and $\displaystyle \colim_{k} \Sigma^{-k} (S^k \otimes Com)$  are cofibrant replacements of $\Ss(1)$.  This allows us to deduce that there is a weak equivalence of right $Com$-modules
$\displaystyle TQ(I) \simeq \colim_k (\Sigma^{-k} S^k \otimes I)$,
recovering the result that the functor $TQ$ is stabilization in $\ComAlg$.

In \secref{tower sec} we review some results from \cite{kuhn pereira}. We note that a right $Com$-module $M$ admits a canonical decreasing filtration
$$ M = F^1M \la F^2M \la F^3M \la \dots $$
such that for all $Com$-algebras $I$, there is a homotopy fiber sequence
$$ F^{n-1}M \ca^L I \ra F^nM \ca^L I \ra M(n) \sm_{\Sigma_n} TQ(I)^{\sm I}.$$
If we let $I^n = (F^nCom) \ca^L I$, this specializes to define a natural `augmentation ideal filtration' of a $Com$--algebra $I$, 
$$ I = I^1 \la I^2 \la I^3 \la \dots,$$ 
with homotopy fiber sequences 
$$I^{n-1} \ra I^n \ra \Ss \sm_{\Sigma_n} TQ(I)^{\sm n}.$$

In \secref{com module filt sec} we show that a right $Com$--module $M$ admits a canonical increasing filtration by right $Com$--modules
$$  F_1M \ra F_2M \ra F_3M \ra \dots$$
such that $F_nM(k) = M(k)$ if $n\geq k$, and for all $Com$--algebras $I$, there is a homotopy fiber sequence
$$ F_{n-1}M \circ_{Com}^L I \ra F_nM \circ_{Com}^L I \ra \bar M(n)\sm_{\Sigma_n}I^n,$$
where $\bar M(n)$ is the cofiber of $F_{n-1}M(n) \ra M(n)$.

In \secref{Kuhn filt sec}, we apply the theory of the previous section to compare two filtrations of $TQ_{\Com}(I)$ that are in the literature. We identify $\bar M(n)$ when $M = K \otimes Com$, recovering the filtrations in \cite{K tensor R} of $K \otimes I$, and then $TQ(I) \simeq \colim_n \Sigma^{-n} S^n \otimes I$.  Similarly, we show that $\bar M(n)= Lie(n)$ when $M = B(S(1), Com, Com)$, where $Lie$ is the Lie cooperad, recovering the filtration of $TQ_{Com}(I)\simeq B(S(1), Com, Com) \circ_{Com} I$ described in \cite{behrens rezk}.  The naturality of our filtration applied to the equivalence of right $Com$--modules
$$ \colim_n \Sigma^{-n} (S^n \otimes Com) \simeq B(S(1),Com,Com)$$
then implies that these two filtrations of $TQ_{Com}(I)$ agree up to homotopy, as had been conjectured.

A final short section, \secref{remarks sec}, has some remarks about generalizations of these results with the operad $Com$ replaced by a more general operad $\Oo$. It is also certainly true that our underlying category of spectra could  be replaced by many different stable symmetric monoidal $\infty$--categories with analogous results.  For simplicity, we have not pursued this in this survey.  For similar reasons, we are a bit informal regarding model category details: for model category details in our setting see \cite{fresse, harper bar construction, harper and hess, kuhn pereira, pereira hha 2016 }.

\section{Background} \label{background sec}

\subsection{Operads, modules, and algebras} \label{OUR SETTING}

In this paper the category $Sp$ will mean the category of symmetric spectra as defined in \cite{hss}.

With the smash product $\sm$ as tensor product and sphere spectrum $\Ss$ as unit, $Sp$ is a closed symmetric monoidal category. It is also tensored over the category $\T$ of based simplicial sets (the `spaces' of the introduction) with $K \otimes X = K \sm X$, for $K \in \T$ and $X \in Sp$.  There is a notion of weak equivalence, and various model structures on $Sp$ compatible with these, such that the resulting quotient category models the standard stable homotopy category.  For our purposes we use the positive model structure used in \cite{kuhn pereira}.

Let $\SnSp$ denote the category of spectra equipped with an action of the $n$th symmetric group $\Sigma_n$.  We then let $\Sym(Sp)$ be the category of symmetric sequences in $Sp$: sequences
$M = \{M(0),M(1),M(2),\dots\}$, with $M(n) \in \SnSp$ for all $n$.

There are embeddings $i_n: \SnSp \ra \Sym(Sp)$ given by letting $i_n(X) = (*, \dots, *, \overset{n}{X}, *, \dots)$.  We will sometimes identify $X \in Sp$ with $i_0(X) = (X,*,*,\dots) \in \Sym(Sp)$.

The category of symmetric sequences in $Sp$, $\Sym(Sp)$, admits a symmetric monoidal product $\otimes$ and a nonsymmetric monoidal `composition' product $\circ$. We recall how this goes.

Given $M, N \in \Sym(Sp)$, $M \otimes N$ is defined by 
$$ (M \otimes N)(n) = \bigvee_{l+m = n} \Sigma_{n+} \sm_{\Sigma_l \times \Sigma_m}(M(l) \sm N(m)),$$
and then $M \circ N$ is defined by
$$ (M \circ N)(n) = \bigvee_{r = 0}^{\infty} M(r) \sm_{\Sigma_r}N^{\otimes r}(n).$$
 
For $n \geq 0$, let ${\bf n} = \{1, \dots, n\}$. The formula for $(M \circ N)(n)$ rewrites as 
\begin{equation*}\label{COMPPROD EQ}
 (M \circ N)(n) = \bigvee_{r=0}^{\infty} M(r) \sm_{\Sigma_r} \left(\bigvee_{\phi: \bf n \ra \bf r} N(\phi)\right),
\end{equation*}
where $N(\phi) = N(n_1(\phi)) \sm \dots \sm N(n_r(\phi))$ and $n_k(\phi)$ is the cardinality of $\phi^{-1}(k)$.

It is useful to make some first observations.
\begin{lem} \label{circ properties lem}
(a) $\displaystyle (M\circ N)(0) = \bigvee_{r = 0}^{\infty} M(r) \sm_{\Sigma_r} N(0)^r$

(b) If $N(n) = *$ for all $n \neq 0$, then, for all $M$, $(M \circ N)(n) = *$ for all $n \neq 0$.

(c) If $N(0) = *$, then  $\displaystyle (M \circ N)(n) = \bigvee_{r=1}^{\infty} M(r) \sm_{\Sigma_r} \left(\bigvee_{\phi \in \Epi(\bf n, \bf r)} N(\phi)\right)$,
    where $\Epi(\bf n, \bf r)$ denotes the set of epimorphisms from $\bf n$ to $\bf r$.
    
(d) $M \circ N$ commutes with all colimits in the first variable, and with filtered colimits and reflexive equalizers in the second variable. 
\end{lem}
A reference for this last property is \cite[Prop. 2.4.1]{fresse}

The category $(\Sym(Sp), \circ, \Ss(1))$ is monoidal, with unit $\Ss(1) = i_1(\Ss) = (*,\Ss,*,*, \dots)$. 

An operad $\Oo$ is then a monoid in this monoidal category. The operad unit $\Ss(1) \ra \Oo$ corresponds to a map $\Ss \ra \Oo(1)$, and the operad multiplication $$ \mu: \Oo \circ \Oo \ra \Oo$$ corresponds to a family of appropriately compatible multiplication maps
$$\mu_{\phi}:\Oo(r) \sm \Oo(\phi) \ra \Oo(n)$$
for all $\phi: \bf n \ra \bf r$.

Similarly, left and right $\Oo$--modules $N$ and $M$ have multiplication maps $\Oo \circ N \ra N$ and $M \circ \Oo \ra M$.  An $\Oo$--algebra is a left $\Oo$--module concentrated in degree 0, and can be viewed as a spectrum $I$ equipped with appropriately compatible maps $\Oo(n) \sm_{\Sigma_n} I^{\sm n} \ra I$.

Finally we recall that if $M$ is a right $\Oo$--module, and $N$ is a left $\Oo$--module, $M \co N$ is defined as the coequalizer in $\Sym(Sp)$ of the two evident maps
$ M \circ \Oo \circ N \begin{array}{c} \longrightarrow \\[-.1in] \longrightarrow
\end{array} M \circ N$.  One has isomorphisms $M \co \Oo = M$ and $\Oo \co N = N$.
Extra structure on $M$ or $N$ can then induce evident extra structure on $M \co N$. In particular, if $I$ is an $\Oo$--algebra, and $M$ is an $\Oo$--bimodule, then $M \co I$ is again an $\Oo$--algebra.

\subsection{$Com$ modules and algebras}

We specialize the above to the case when $\Oo$ is the reduced commutative operad $Com$, where $Com$ has underlying symmetric sequence $(*,\Ss,\Ss,\Ss, \dots)$ and structure maps $\mu_{\phi} = id: \Ss \ra \Ss$ for all epimorphisms $\phi: \bf n \ra r$ with $r \geq 1$.

We let $\MCom$, $\ComM$, and $\ComAlg$ respectively denote the categories of right $Com$-modules, left $Com$-modules, and $Com$-algebras.  The category $\ComAlg$ is easily seen to be equivalent to the category of non-unital commutative algebras in $Sp$.

We will require that our right $Com$--modules $M$ be {\em reduced}: $M(0) = *$.   A right $Com$-module $M$ then  consists of appropriately compatible maps, for all epimorphisms $\phi: \bf n \ra \bf r$ with $r \geq 1$,
$$ \Delta_{\phi}: M(n) \ra M(r).$$
Thus $\MCom$ is equivalent to the category of functors $M: \Epi^{op} \ra \Sym(Sp)$, where $\Epi$ is the category with objects $\bf n$, $n \geq 1$, with morphisms $\Epi(\bf n, \bf r)$ as above.

A right $Com$--module $M$ will be a $Com$--bimodule if one also has appropriately compatible maps
$ \mu_{\phi}: N(\phi) \ra N(n)$  for all epimorphisms $\phi: \bf n \ra \bf r$. In this case, if $I$ is a $Com$--algebra, so is $M \circ_{Com} I$.

\begin{exs}  Here are some examples of $Com$--algebras.

(a) Given $X \in Sp$, $Com \circ X$ can be viewed as the free commutative nonuntial algebra generated by $X$.  Explicitly, 
$$ (\Com \circ X)(0) = \bigvee_{n=1}^{\infty} X^{\sm n}_{\Sigma_n}.$$

(b)  If $Z$ is a based space, its Spanier--Whitehead dual, $D(Z)$ is a $Com$--algebra with multiplication induced by the diagonal $\Delta: Z \ra Z \sm Z$.

(c) Given $X \in Sp$, let $I(X)$ be the fiber of the natural map $\Sinfty (\Oinfty X)_+ \ra \Ss$. Then $I(X)$ is a $Com$--algebra and the natural map of spectra $I(X) \ra \Sinfty \Oinfty X$ is a weak equivalence. (See \cite[Thm. 2.13]{kuhn hurewicz} for details.)
\end{exs}

\begin{rem}  This last example is a reminder that the category of non-unital commutative algebras is equivalent to the category of augmented commutative algebras via the correspondence sending an augemented algebras $A \ra \Ss$ to its `augmentation ideal' $I(A)$ defined as the fiber of the augmentation. See \cite{basterra}.
\end{rem}

\begin{exs} \label{M circ Com ex} Given $r \geq 1$, let $P_r$ be the right $Com$--module given by $P_r(n) = \Epi({\bf n,r})_+ \sm \Ss$.  Note that $P_r$ is acted on the left by $\Sigma_r$.  Given $M \in Sym(Sp)$, $M\circ Com$, the free right $Com$-module generated by $M$, satisfies
$$ (M \circ Com) = \bigvee_{r = 1}^{\infty} M(r)\sm_{\Sigma_r} P_r.$$
\end{exs}

\begin{ex}  The natural map of operads $Com \ra \Ss(1)$ makes $\Ss(1)$ into a $\Com$-bimodule.
\end{ex}

\subsection{Model structures}\label{MODELSTRUCTURE SEC}

We start with the model structure on symmetric spectra $Sp$ as in \cite{hss}, and then give $\Sym(Sp)$ the associated injective model structure: $f: M \ra N$ in $\Sym(Sp)$ is a weak equivalence or cofibration if it is levelwise a weak equivalence or cofibration.  We note that since $\Ss$ is cofibrant in $Sp$, both $\Ss(1)$ and $Com$ are cofibrant in $\Sym(Sp)$.

With $U$ denoting the evident forgetful functor, the adjoint pair $$ \circ \ Com: \Sym(Sp)\begin{array}{c} \longrightarrow \\[-.1in] \longleftarrow 
\end{array} \MCom: U$$ induces a model structure on $\MCom$ in a standard way: maps in $\MCom$ are weak equivalences or fibrations if they are when regarded as maps in $\Sym(Sp)$, and a generating set of cofibrations is given by the maps  of the form $M \circ Com \ra N \circ Com$, where $M \ra N$ is a cofibration in $\Sym(Sp)$.  

We note that, from the description of $M \circ Com$ in \exref{M circ Com ex}, one sees that a cofibration in $\MCom$ remains a cofibration when regarded as a map in $\Sym(Sp)$.  In particular, if $M$ is cofibrant in $\Sym(Sp)$, then so is $M \circ Com$.

Finally, we give $\ComAlg$ the model structure used in \cite{kuhn pereira}. We will need to know that $I \ra J$ is a weak equivalence in $\ComAlg$ if it is a weak equivalence in $Sp$, and that, if $I$ is cofibrant in $\ComAlg$, then it is cofibrant in $Sp$.  Given $I \in \ComAlg$ we let $I^c$ denote a cofibrant replacement.

\section{The bar construction and $TQ(I)$} \label{bar construction and TQ section}

Given $M \in \MCom$ and $N \in \ComM$, $B(M,Com,N) \in \Sym(Sp)$ is defined as the geometric realization of the simplicial object $B_*(M,Com,N)$ in $\Sym(Sp)$ defined by
$ B_s(M,Com,N) =  M \circ Com^{\circ s} \circ N$.

Note that if $M$ is a $Com$-bimodule and $I$ is a $Com$-algebra, then \\ $B(M,Com,I)$ is again a $Com$-algebra.

\begin{prop}  (a) $B(M,Com,N) = B(M,Com,Com) \circ_{Com} N$.

(b) The natural map $\epsilon: B(M,Com,Com) \ra M$ is a weak equivalence in $\MCom$.

(c) If $M$ is cofibrant in $\Sym(Sp)$, then $B(M,Com,Com)$ is a cofibrant right $Com$--module.
\end{prop}
\begin{proof}[Sketch Proof]  The first statement follows from the fact that $M \circ N$ commutes with all colimits in the variable $M$.  

The proof of the second statement is standard. One lets $\underline{M}_*$ be the constant simplicial object in $\Sym(Sp)$ with $\underline{M}_s = M$ for all $s$, so that $|\underline{M}_*| = M$.  Then one notes that there is a map of simplical objects $\eta_*: \underline{M}_* \ra B_*(M,Com,Com)$ such that $\epsilon_* \circ \eta_* = 1_{\underline{M}_*}$, and a simplicial homotopy from $\eta_* \circ \epsilon_*$ to $1_{B_*(M,Com,Com)}$.  Now apply realization.

For the last statement, $M$ being cofibrant in $\Sym(Sp)$ implies that $M \circ Com^{\circ s}$ will also cofibrant in $\Sym(Sp)$, and thus that $B_s(M,Com,Com) = (M \circ Com^{\circ s}) \circ Com$ will be cofibrant in $\MCom$. The geometric realization of a simplicial set of cofibrant objects will again be cofibrant, so statement (c) follows.
\end{proof}

\begin{prop}\cite[Prop.2.9]{kuhn pereira} If $M \in \MCom$ is cofibrant in $\Sym(Sp)$, and 
 $I \in \ComAlg$ is cofibrant, the natural map $B(M,Com,I) \ra M \circ_{Com} I$ is a weak equivalence.  
\end{prop}

We also record the following useful properties.

\begin{prop} \cite[Thm.2.11]{kuhn pereira} \label{kp prop} \  If we restrict right $Com$-modules to those which are cofibrant in $\Sym(Sp)$, and $Com$-algebras to those which are cofibrant in $Sp$,  the following properties hold.

(a) $B(M,Com,I)$ takes weak equivalences in either variable to weak equivalences of spectra.

(b) A homotopy cofibration sequence in $\MCom$  $ L \ra M \ra N$ induces a homotopy cofibration sequence of spectra
$ B(L,Com,I) \ra B(M,Com,I) \ra B(N,Com,I)$ for all $Com$--algebras $I$.
\end{prop}

The following definition of `the spectrum of derived indecomposables of a commutative ring spectrum' goes back to \cite{basterra}, with antecedents in the work of Andr\'e and Quillen.   

Note that $\Ss(1)$ is a $Com$-bimodule, via the map of operads $Com \ra \Ss(1)$, and then that $\Ss(1) \circ_{Com}I$ identifies as the pushout $I/I^2$ in the diagram
\begin{equation*}
\SelectTips{cm}{}
\xymatrix{
 I \sm I  \ar[d] \ar[r] & I \ar[d]  \\
{*} \ar[r] & I/I^2, }
\end{equation*}

\begin{defn} Given $I \in \ComAlg$, let $TQ(I) = B(\Ss(1),Com, I^c)$.
\end{defn}

 $TQ(I)$ can by viewed as the left derived circle product $\Ss(1) \ca^L I$, where we have used $B(\Ss(1), Com, Com)$ as our cofibrant replacement for $\Ss(1)$ in $\MCom$.  In \secref{taq is stablization sec}, we will find an alternative cofibrant replacement.

\section{A model for $K \otimes I$ using $Com$--bimodules} \label{tensor sec}

The category $\ComAlg$ is tensored and cotensored over the category $\T$ of based simplicial sets: given $K \in \T$, there is a natural isomorphism
$$ \Map_{\ComAlg}(K \otimes I, J) \simeq \Map_{\T}(K, \Map_{\ComAlg}(I,J)).$$

In this section, we show that $K \otimes I$ has the form $M \circ_{\Com}I$, with $M$ a cofibrant right $Com$-module that is a functor of $K$.  This is an elaboration of old observations of ours in \cite{K tensor R}, but is presented in a way that better allows for generalizations.

\begin{defn}  Given $K \in \T$, define a $Com$-bimodule $(K \otimes Com)$ as follows. 

Let $(K \otimes Com)(n) = K^{\sm n} \sm \Ss$.  
Given an epimorphism $\phi: \rr \ra \n$, one has an associated diagonal map
$$ \Delta_{\phi}: K^{\sm n} \ra K^{\sm r}$$ 
and an evident homeomorphism
$$ \mu_{\phi}: K^{|\phi^{-1}(1)|} \sm \cdots \sm K^{|\phi^{-1}(n)|} \xra{\sim} K^r.$$
The maps $\Delta_{\phi}$ induce a right $Com$-module structure on $(K \otimes Com)$, while the maps $\mu_{\phi}$ induce a left $Com$-module structure.
\end{defn}

\begin{defn}  Define a natural algebra map $\Phi_{K,I}: K \otimes I \ra (K\otimes Com)\circ_{Com} I$ to be the adjoint of the composite
\begin{equation*}
\begin{split}
K = \Map_{\T}(S^0,K) & \ra \Map_{Com\text{-bimod}}(S^0 \otimes Com, K \otimes Com) \\ 
 & \ra \Map_{\ComAlg}((S^0 \otimes Com) \circ_{Com}I, (K \otimes Com)\circ_{Com}I) \\
 & = \Map_{\ComAlg}(I, (K \otimes Com)\circ_{Com}I).
\end{split}
\end{equation*}
Note that we have used that $(S^0 \otimes Com) = Com$, so that $(S^0 \otimes Com) \circ_{Com}I = I$.
\end{defn}

\begin{thm} \label{tensor thm} $\Phi_{K,I}: K \otimes I \ra (K\otimes Com)\circ_{Com} I$ is an isomorphism in $\ComAlg$.
\end{thm}
\begin{proof}  We first check this when $I = Com \circ X$, with $X \in \T$. 

The category $Sp$ is tensored over $\T$, with $K \sm X$ as the tensor product of $K \in \T$ and $X \in \Sp$.  It follows formally then that $K \otimes (Com \circ X) = Com \circ (K \sm X)$.  Thus
\begin{equation*} 
\begin{split}K \otimes (Com \circ X) &= Com \circ (K \sm X) = \bigvee_{n = 1}^{\infty} (K \sm X)^{\sm n}_{\Sigma_n} = \bigvee_{n = 1}^{\infty} K^{\sm n} \sm_{\Sigma_n} X^{\sm n}\\  & = (K \otimes Com) \circ X
 = (K \otimes Com) \circ_{Com} (Com \circ X),
\end{split}
\end{equation*}
and one sees that $\Phi_{K, Com \circ X}$ is a natural isomorphism for all $X$.

As functors on $\ComAlg$, $K \otimes \text{\underline{\hspace{.1in}}}$ commutes with all colimits, and a consequence of \lemref{circ properties lem}(d) is that $(K \otimes Com) \circ_{Com} \text{\underline{\hspace{.1in}}}$ commutes with coequalizers that are reflexive in $\Sym(Sp)$.

Any $I \in \ComAlg$ is the coequalizer of $Com \circ Com \circ I \rightrightarrows Com \circ I$ which is reflexive in $\Sym(Sp)$.  Since $\Phi_{K, Com \circ I}$ and $\Phi_{K,Com \circ Com \circ I}$ are isomorphisms, it follows that $\Phi_{K,I}$ is also.
\end{proof}

Now we turn to the cofibrancy property of $K \otimes Com$.

\begin{thm} \label{cofibancy of tensor thm}  $K \otimes Com$ is a cofibrant right $Com$-module.
\end{thm}
\begin{proof}  Note that if $K = |K_*|$ then $K \otimes Com = |K_* \otimes Com|$, and thus it suffices to proof the theorem when $K$ is a based set.  As any based set is the filtered colimit of finite based sets of the form $\m_+$, it suffices to check the theorem when $K = \m_+$.  

Now we observe that maps have epi-mono factorizations:
$$ \m^n = \Map_{Sets}(\n,\m) = \coprod_{r} \Mono(\rr, \m) \times_{\Sigma_r} \Epi(\n,\rr)$$
so that 
$$(\m_+ \otimes Com)(n) = (\m^n)_+ \sm \Ss = \bigvee_{r} \Mono(\rr,\m)_+ \sm_{\Sigma_r} P_r(n) = (\Mono_m \circ Com)(n),$$
where $\Mono_m \in \Sym(Sp)$ is given by $\Mono_m(r) = \Mono(\rr,\m)_+ \sm \Ss$. The cofibrancy of $\Mono_m$  in $\Sym(Sp)$ then implies that $(\m_+ \otimes Com) = \Mono_m \circ Com$ is a cofibrant right $Com$-module.

\end{proof}

\section{$TQ(I)$ as stablization.}  \label{taq is stablization sec}

\begin{lem} In $Sp$, there are weak equivalences 
$\displaystyle \colim_k \Sigma^{-k} S^{kn} \sm \Ss \simeq
\begin{cases}  \Ss \text{ if } n= 1 \\ * \text{ if } n > 0
\end{cases}
$.
\end{lem}
This implies the following.

\begin{prop} \label{suspension model prop} There is a weak equivalence of right $Com$-modules $$\displaystyle \colim_k \Sigma^{-k}(S^k \otimes Com) \xra{\sim} \Ss(1).$$
\end{prop}

The two weak equivalences of right $Com$--modules 
$$  B(\Ss, Com, Com) \xra{\sim} \Ss(1) \xla{\sim} \colim_k \Sigma^{-k}(S^k \otimes Com)$$
now combine to show the next result.

\begin{cor} $\displaystyle TQ(I) \simeq \colim_k \Sigma^{-k} (S^k \otimes I)$
\end{cor}

We have recovered the result that the functor $TQ$ is stabilization in $\ComAlg$.  This result is well known -- see \cite{basterra mccarthy} -- but our proof here, based on the evident lemma, seems simpler and more transparent than other proofs. 

\section{A canonical decreasing filtration of right $Com$-modules} \label{tower sec}

A right $Com$-module $M$ has a natural decreasing filtration
$$ M = F^1M \la F^2M \la F^3M \la \dots$$
defined in an elementary way.

\begin{defn}  Given $M \in \MCom$, $F^nM \in \MCom$ is defined by letting
$$(F^nM)(r) = 
\begin{cases} 
* & \text{if } r < n \\ M(n) & \text{if } r \geq n. 
\end{cases}
$$
\end{defn}

Let $\bar F^nM$ denote the cofiber of $F^{n+1}M \ra F^nM$.  The following is clear by inspection.

\begin{lem}  $\bar F^nM = i_n(M(n)) \circ \Ss(1)$, as right $Com$-modules.
\end{lem}

\begin{prop} \label{decreasing prop} If $I$ is cofibrant in $\ComAlg$ and $M \in \MCom$ is cofibrant in $\Sym(Sp)$, there is a decreasing filtration on $B(M,Com,I)$, and fibration sequences
$$ F^{n-1}B(M,Com,I) \ra F^nB(M,Com,I) \ra M(n) \sm_{\Sigma_n} TQ(I)^{\sm n}.$$
\end{prop}   
\begin{proof}  Combining the lemma and \propref{kp prop}, the cofiber of 
$$F^{n-1}B(M,Com,I) \ra F^nB(M,Com,I)$$ is $B(i_n(M(n)) \circ \Ss(1), Com, I)$. We rewrite this:
\begin{equation*}
\begin{split} 
B(i_n(M(n)) \circ \Ss(1), Com, I) & = i_n(M(n)) \circ B(\Ss(1), Com, I) \\
   & = i_n(M(n)) \circ TQ(I) = M(n) \sm_{\Sigma_n} TQ(I)^{\sm n}.
\end{split}
\end{equation*}
\end{proof}

Note that the decreasing filtration of $\Com$ is a filtration of $Com$-bimodules. If we let $I^n = B(F_nCom,Com,I)$, the proposition has the following corollary.

\begin{cor} If $I$ is cofibrant in $\ComAlg$, There is a decreasing filtration $I \simeq I^1 \ra I^2 \ra I^3 \ra \dots$, and fibration sequences 
$$ I^{n-1} \ra I^n \ra Com(n) \sm_{\Sigma_n} TQ(I)^{\sm n}.$$
\end{cor}

\begin{rem} In \cite{kuhn pereira}, the functors $I \mapsto I^n$ are shown to have lots of nice additional structure; in particular, there are natural maps $(I^m)^n \ra I^{mn}$ induced by the operad structure on $Com$.
\end{rem}

\begin{rem} Let $p_nB(M,Com,I)$ be the cofiber of $B(F_{n+1}M,Com,I) \ra B(M,Com,I)$, so there is a tower under $B(M,Com,I)$:
$$  \cdots \ra p_3B(M,Com,I) \ra p_2B(M,Com,I) \ra p_1B(M,Com,I).$$
\propref{decreasing prop} and the fact that $TQ(I)$ is stablization of $\ComAlg$ implies that this is the Goodwillie tower of the functor $$B(M,Com,\text{\underline{\hspace{.1in}}}): \ComAlg \ra Sp.$$ Similarly, if $p_nI = B(Com/F_{n+1}Com, Com, I)$ then $I^{n+1} \ra I \ra p_nI$ is a homotopy fibration sequence in $\ComAlg$, and the resulting tower under $I$,
$$  \cdots \ra p_3I \ra p_2I \ra p_1I,$$
identifies with the Goodwillie tower of the identity functor on $\ComAlg$.
See \cite{kuhn pereira} for a bit more discussion of this.  Towers like this were also studied in \cite{harper and hess}.
\end{rem}

\section{A canonical increasing filtration of right $Com$--modules} \label{com module filt sec}

In this section, we define a natural increasing filtration of the category $\MCom$, and develop some useful properties.

Recall that $i_n: \Sigma_n\text{-}Sp \ra \Sym(Sp)$ is the functor sending a spectrum $X$ with $\Sigma_n$-action to the symmetric sequence which is $X$ at level $n$ and $*$ at all other levels.  Given any $M \in \Sym(Sp)$, one has a natural inclusion $i_n(M(n)) \ra M$.  For $M \in \MCom$, this inclusion induces a natural map in $\MCom$
$$ e_n: i_n(M(n))\circ Com \ra M.$$
Finally, we recall that $i_n(M(n))\circ Com = M(n)\sm_{\Sigma_n}P_n$, where $P_n(r) = \Epi(\rr,\n)_+ \sm \Ss$.

\begin{defn}  Given $M \in \MCom$, we define an increasing filtration of $M$ by right $Com$-modules, 

\begin{equation*}
\SelectTips{cm}{}
\xymatrix{
F_1M \ar[dr]_-<<{g_1} \ar[r]^{f_1} & F_2M \ar[d]^-{g_2} \ar[r]^{f_2} & F_3M \ar[r]^{f_3} \ar[dl]^<<{g_3} & \dots \\
&M & &  }
\end{equation*}
as follows.

Let $F_1M = i_1(M(1)) \circ Com$, with $g_1 = e_1: i_1(M(1))\circ Com \ra M$.  

Suppose that $n>1$, and we have defined $g_{n-1}:F_{n-1}M \ra M$.

The map $g_{n-1}$ induces a map $g^n_{n-1}:i_n((F_{n-1}M)(n)) \circ Com \ra i_n(M)(n))\circ Com$.  

We now construct the following commutative diagram: 
\begin{equation} \label{filtration diagram}
\SelectTips{cm}{}
\xymatrix{
i_n((F_{n-1}M)(n)) \circ Com \ar[d]_{g_{n-1}^n} \ar[r]^-{e_n} & F_{n-1}M \ar[d]_{f_{n-1}} \ar@/^1pc/[ddr]^{g_{n-1}}  & \\
i_n(M(n)) \circ Com \ar[r] \ar@/_/[rrd]^-{e_n} & F_nM \ar[dr]^-{g_n} & \\
& & M.}
\end{equation}
In this diagram, the inner square is defined to be a pushout, defining $F_nM$ and $f_{n-1}$. Since the outer square commutes, there is an induced map $g_n$ making the diagram commute. 

This completes the recursive step of our definition..
\end{defn}

The next proposition implies that $\displaystyle \colim_n F_nM = M$.

\begin{prop}  $g_n: (F_nM)(r) \ra M(r)$ is an isomorphism if $n\geq r$.
\end{prop}
\begin{proof}  We fix $r$ and prove this by induction on $n$,

We first check this when $n=r$.  Diagram (\ref{filtration diagram}) evaluated on $n$ yields the commutative diagram
\begin{equation*}
\SelectTips{cm}{}
\xymatrix{
(F_{n-1}M)(n) \ar[d] \ar@{=}[r] & (F_{n-1}M)(n) \ar[d]_{f_{n-1}} \ar@/^1pc/[ddr]^{g_{n-1}}  & \\
M(n) \ar[r] \ar@{=}[rrd] & (F_nM)(n) \ar[dr]^-{g_n} & \\
& & M(n).}
\end{equation*}
Since the inner square is a pushout, it follows that $g_n$ is an isomorphism.

Now suppose that $r<n$ and we have proved that $g_{n-1}: (F_{n-1}M)(r) \ra M(r)$ is an isomorphism .  Diagram (\ref{filtration diagram}) evauated on $r$ yields the commutative diagram
\begin{equation*}
\SelectTips{cm}{}
\xymatrix{
(F_{n-1}M)(n) \sm_{\Sigma_n} P_n(r) \ar[d] \ar@{=}[r] & (F_{n-1}M)(r) \ar[d]_{f_{n-1}} \ar@/^1pc/[ddr]^{g_{n-1}}  & \\
M(n) \sm_{\Sigma_n} P_n(r) \ar[r] \ar@/_/[rrd] & (F_nM)(r) \ar[dr]^-{g_n} & \\
& & M(r).}
\end{equation*}
Since $r<n$, $P_n(n) = *$ (since $\Epi(\rr,\n) = \emptyset$), so that the left vertical map in the pushout square has the form $* \ra *$. We conclude that the right vertical map $f_{n-1}$ is an isomorphism, and thus so is $g_n$.
\end{proof}

Our next results implies that our filtration is homotopically well behaved when $M$ is a cofibrant right $Com$-module.

\begin{prop} \label{cofib filt prop} Given $M \in \MCom$, the following statements are equivalent.

(a) $M(1)$ is cofibrant in $Sp$, and each map $(F_{n-1}M)(n) \ra M(n)$ is a cofibration in $\Sigma_n\text{-}Sp$.

(b) $F_1M$ is cofibrant in $\MCom$,  and each map $F_{n-1}M \ra F_nM$ is a cofibration between cofibrant objects in $\MCom$.

(c) $M$ is cofibrant in $\MCom$.
\end{prop}
\begin{proof}  Assume that (a) holds. It then follows that $F_1M$ is cofibrant in $\MCom$, and that, for each $n>1$, the left vertical map in the pushout square in (\ref{filtration diagram}) is a cofibration in $\MCom$.  Thus, so is the right vertical map, i.e. $f_{n-1}:F_{n-1}M \ra F_nM$ is a cofibration in $\MCom$, and (b) is true.  

The statement that (b) implies (a) follows from the fact that if $M \ra N$ is a cofibration in $\MCom$, then it is a cofibration when viewed in $\Sym(Sp)$.

The statement that (b) $\Rightarrow$ (c) follows from the previous proposition.

Finally, we show that (b) holds if $M$ is cofibrant in $\MCom$.  First consider the special case when $M = i_m(X) \circ Com = X \sm_{\Sigma_m} P_m$ with $X$ a cofibrant object in $\Sigma_m\text{-}Sp$.  By inspection, one sees that the filtration on $X \sm_{\Sigma_m} P_m$ is very simple: 
\begin{equation*}
F_n(X \sm_{\Sigma_m} P_m) = 
\begin{cases}
* & \text{if } n<m \\ X \sm_{\Sigma_m} P_m & \text{if } n \geq m,
\end{cases}
\end{equation*}
and one sees that (b) is true.  

By construction, $F_n$ commutes with pushouts and filtered colimits.  Since an any cofibrant right $\Com$-module $M$ is built by these operations from modules as in the special case, we see that (b) will hold for any cofibrant right $\Com$--module.
\end{proof}

\begin{thm} \label{fitration equiv thm} If $f: M \ra N$ is a weak equivalence between cofibrant right $Com$-modules, then, for all $n$,  $F_nf: F_nM \ra F_nN$ will also be a weak equivalence between cofibrant right $\Com$-modules.
\end{thm}
\begin{proof} We prove this by induction on $n$.  Since $f$ is a weak equivalence, so is $f(1): M(1) \ra N(1)$, and thus so is $F_1f: M(1) \sm P_1 \ra N(1) \sm P_1$.

Assume that $F_{n-1}f$ is a weak equivalence.  One has a commutative cube in $\MCom$:
\begin{equation*}
\xymatrix {
& F_{n-1}M(n)\sm_{\Sigma_n}P_n \ar@{->}[dl]^-{F_{n-1}f(n)} \ar@{->}'[d][dd] \ar[rr] &&  F_{n-1}M \ar@{->}[dl]^-{F_{n-1}f} \ar@{->}[dd]  \\
F_{n-1}N(n) \sm_{\Sigma_n} P_n \ar@{->}[dd] \ar[rr] & & F_{n-1}N \ar@{->}[dd]&  \\
& M(n)\sm_{\Sigma_n}P_n \ar@{->}[dl]^-{f(n)}\ar@{->}'[r][rr]  &&  F_nM \ar@{->}[dl]^-{F_nf}  \\
N(n) \sm_{\Sigma_n}P_n \ar[rr] & & F_nN,&
}
\end{equation*}
with front and back faces pushouts. By the last proposition, the vertical maps are cofibrations, so the front and back faces are homotopy pushouts.   By assumption, the map $F_{n-1}f$ is a weak equivalence, and thus so is the map labelled $F_{n-1}f(n)$. Since $f$ is a weak equivalence, so is the map labelled $f(n)$.  We conclude that $F_nf$ is also an equivalence.
\end{proof}

We end this section with one more observation. Let  $\bar F_nM$ denote the cofiber of $F_{n-1}M \ra F_nM$, and let $\bar M(n)$ denote the cofiber of $(F_{n-1}M)(n) \ra M(n)$. Note that $\bar M(n) = (\bar F_nM)(n)$.

\begin{cor} \label{cofiber cor}  If $M$ is cofibrant in $\MCom$, then $\bar F_nM = i_n(\bar M(n)) \circ Com$.  Thus, if $I$ is cofibrant in $\ComAlg$, $B(M,\Com, I)$ is filtered with cofiber quotients
$$ \bar F_nB(M,Com,I) \simeq \bar M(n) \sm_{\Sigma_n} I^{\sm n}.$$ 
\end{cor} 
\begin{proof}  The first statement follows from \propref{cofib filt prop}.   We check the last statement: 
\begin{equation*}
\begin{split}
\bar F_nB(M,Com,I) & = B(i_n(\bar M(n)) \circ Com, Com, I) \\
 & = i_n(\bar M(n))\circ B(Com,Com,I) \\ & \simeq i_n(\bar M(n)) \circ I = \bar M(n) \sm_{\Sigma_n}I^{\sm n}.
\end{split}
\end{equation*}
\end{proof}

\section{Filtrations on $TQ(I)$ and $K \otimes I$} \label{Kuhn filt sec}

Lets see what our increasing filtration looks like for our featured right $Com$-modules.

\begin{lem} For any $M \in \Sym(Sp)$, $\displaystyle F_n(M \circ Com) = \bigvee_{r=1}^n i_r(M(r)) \circ Com$.  Thus 
$\bar F_n(M \circ Com)=i_n(M(n)) \circ Com = M(n) \sm_{\Sigma_n}P_n$, 
and so 
$\bar F_n(M \circ Com)(n) = M(n)$.
\end{lem}
\begin{proof} This is easy to see, noting that $\displaystyle M \circ Com = \bigvee_{r+1}^{\infty} i_r(M(r) \circ Com$.
\end{proof}

\begin{prop} \label{filt of B(M,C,C) prop}  For any $M \in \MCom$, 
$ \bar F_nB(M,Com,Com)(n) = B(M,Com,S(1))(n)$.
\end{prop}
\begin{proof} Since $B(M,Com,Com) = |B_*(M,Com,Com)|$, it follows that 
$$\bar F_nB(M,Com,Com)(n) = |\bar F_nB_*(M,Com,Com)(n)|.$$ Now we compute:
\begin{equation*}
\begin{split}
\bar F_nB_s(M,Com,Com)(n) &
= \bar F_n(M \circ Com^{\circ s} \circ Com)(n)  \\
  & = (M \circ Com^{\circ s})(n) \text{ \ (by the lemma)} \\
  &= B_s(M,Com,\Ss(1))(n).
\end{split}
\end{equation*}

\end{proof}

\begin{cor} If $I$ is cofibrant in $\ComAlg$ and $M \in \MCom$ is cofibrant in $\Sym(Sp)$, then $B(M,Com,I)$ is filtered with cofiber quotients
$$ \bar F_nB(M,Com,I) \simeq B(M,Com,\Ss(1))(n) \sm_{\Sigma_n} I^{\sm n}.$$
\end{cor}

We specialize this to the case when $M=\Ss(1)$.  Michael Ching \cite{ching} showed that the spectra $B(S(1),Com,S(1))(n)$ is the $n$th spectrum in a cooperad that should be called $Lie$. 

Collecting our various results, we have recovered the filtration of $TQ(I)$ as developed in \cite{behrens rezk}.

\begin{thm} \label{TQ filt thm} The right $Com$-module $B(S(1),Com,Com)$ is filtered with cofiber quotients
$$\bar F_nB(S(1),Com,Com)(n) = Lie(n).$$  Thus if $I$ is cofibrant in $\ComAlg$, $TQ(I)$ is filtered with cofiber quotients
$$ \bar F_nTQ(I) \simeq Lie(n)\sm_{\Sigma_n}I^{\sm n}.$$
\end{thm}

We turn to our analysis of $K \otimes I = (K \otimes Com) \circ_{Com} I$.  

The following is easy to see by inspection.

\begin{lem} \label{K filtration lem} $F_r(K\otimes Com)(n) = F_rK^{\sm n}\sm \Ss$, where 
$$F_rK^{\sm n} = \{(x_1 \sm \dots \sm x_n) \in K^{\sm n} \ | \ \text{at most $r$ of the $x_i$ are distinct} \}.$$ 
\end{lem}

We recover the author's filtration of $K \otimes I$, and then $TQ(I)$, as developed in \cite{K tensor R}.

\begin{thm} \label{tensor filt thm} $\bar F_n (K \otimes Com)(n) = (K^{\sm n}/\Delta_n(K)) \sm \Ss$, where $\Delta_n(K)$ is the fat diagonal. Thus if $I$ is cofibrant in $\ComAlg$, $K \otimes I$ is filtered with cofiber quotients
$$ \bar F_n(K \otimes I) \simeq (K^{\sm n}/\Delta_n(K)) \sm_{\Sigma_n}I^{\sm n}.$$
\end{thm}

\begin{cor} \label{old filt cor} If $I$ is cofibrant in $\ComAlg$, $\displaystyle (\colim_k \Sigma^{-k}(S^k \otimes I))$ is filtered with cofiber quotients
$$ \bar F_n((\colim_k \Sigma^{-k}(S^k \otimes I)) \simeq (\colim_k \Sigma^{-k}(S^{kn}/\Delta_n(S^k))) \sm_{\Sigma_n} I^{\sm n}.$$
\end{cor}

Recall that we have shown that if $I$ is a cofibrant algebra, then  
\begin{equation}  TQ(I) \simeq \colim_k \Sigma^{-k}(S^k \otimes I).
\end{equation}

\thmref{TQ filt thm} recovers the filtration of $TQ(I)$ from \cite{behrens rezk}, while \corref{old filt cor} recovers the filtration on $TQ(I)$ from \cite{K tensor R}.  
In \cite{K tensor R}, it was noted that the $\Sigma_n$--spectrum $\displaystyle \colim_k \Sigma^{-k} S^{kn}$  was equivalent to a $\Sigma_n$-spectrum arising in Goodwillie calculus and soon to be called $Lie(n)$. (This observation is due to Arone, see \cite[\S 7]{arone dwyer}.) The authors of \cite{behrens rezk} rather sensibly ask if the filtration on $TQ(I)$ given in \cite{K tensor R} agrees with that in \cite{behrens rezk}. 

Our study of the category of right $Com$-modules makes it easy answer this question.

\begin{thm}  The filtrations of $TQ(I)$ defined in \cite{K tensor R} and \cite{behrens rezk} agree.
\end{thm}
\begin{proof}  As constructed here, both filtrations arise from the canonical filtrations of cofibrant right $Com$-modules.  We have shown that these modules are weakly equivalent:
$$ B(S(1),Com,Com)  \simeq \colim_k \Sigma^{-k}(S^k \otimes Com).$$
Thus \thmref{fitration equiv thm} tells us that these filtrations agree.
\end{proof}
We recover Arone's observation: 
\begin{cor} $\displaystyle B(\Ss(1),Com,\Ss(1))(n) \simeq \colim_k \Sigma^{-k} S^{kn}/\Delta_n(S^k) \sm \Ss$.
\end{cor}

\section{Replacing $Com$ by a more general operad} \label{remarks sec}

For simplicity and because it is so central, we have been working with the operad $Com$, but much of what we have presented here can be generalized with $Com$ replaced by a more general operad in $Sp$.  As was done in \cite{kuhn pereira}, one should work with reduced operads $\Oo$ such that each $\Oo(n)$ is cofibrant in $Sp$, and $S \ra \Oo(1)$ is a cofibration. It is also simplifying to insist that $S \ra \Oo(1)$ is a weak equivalence.

With these hypotheses, all our results about $B(M,Com,I)$ also hold for $B(M, \Oo,I)$, as do our general results about the decreasing and increasing filtration of right $\Oo$-modules.  (The decreasing filtration was already studied in this generality in \cite{kuhn pereira}.) In particular, one defines $TQ_{\Oo}(I)$ to be $B(\Oo(1),\Oo,I)$.

One needs more care with discussions regarding modelling $K \otimes I$, with $K$ a based space and $I$ an $\Oo$-algebra. 

If one defines $K \otimes \Oo$ by letting $(K \otimes \Oo)(n) = K^{\sm n} \sm O(n)$, it isn't hard to see that this is an $\Oo$-bimodule, and then that there is a natural map $K \otimes I \ra (K \otimes \Oo) \co I$ that can be proved to be an isomorphism using the same method of proof as in the proof of \thmref{tensor thm}. 

When $\Oo$ is not $Com$, it is not clear that $(K \otimes \Oo)$ is necessarily cofibrant in the category of right $\Oo$-modules. In particular, our proof that $\m_+ \otimes Com$ is cofibrant -- part of the proof of \thmref{cofibancy of tensor thm} -- does not work for more general operads.  Nevertheless it is possible the result is still true. 

 A work-around for some purposes is to just replace $(K \otimes \Oo)$ by a cofibrant replacement in $\MCom$. If $I$ is a cofibrant $\Oo$-algebra, since $(K \otimes \Oo)$ will be cofibrant in $\Sym(Sp)$, we will still be able to deduce that $(K \otimes \Oo)^c \co I \simeq K \otimes I$. Since, as before, there will be a weak equivalence of cofibrant right $\Oo$-modules
$$  B(\Oo(1), \Oo, \Oo) \simeq \colim_k \Sigma^{-k} (S^k \otimes \Oo)^c,$$
we can again conclude that the functor $TQ_{\Oo}$ corresponds to stabilization in the category of $\Oo$-algebras.

For a general operad $\Oo$, there does not seem to be a neat description of the increasing filtration on $K \otimes \Oo$, analogous to the description given in \lemref{K filtration lem}.  As illustrated in \cite{K tensor R}, with $\Oo=Com$, the resulting nice description of the cofibers of $K \otimes I$ is key to various applications.

\section{Acknowledgments} This paper is a summary of ideas presented in two talks in spring 2023: {\em A ring and its topological Andr\'e-Quillen homology: there and back again} given in March, 2024, in the MHTRT online seminar, and {\em A canonical filtration of right modules over an operad, topological Andre-Quillen homology, and the Lie operad} presented at the Conference on Generalised Lie algebras in Derived Geometry held in June 2024 outside of Utrecht.  We thank the organizers for the opportunity to present our thoughts. The author's visit to Utrecht was partially supported by a grant from the Simons Foundation.

%%%%%%%%%% Insert bibliography here %%%%%%%%%%%%%%


\begin{thebibliography}{9}

\bibitem{arone dwyer} G.Z.~Arone and W.G.~Dwyer, {\em Partition complexes, Tits buildings, and symmetric products}, Proc. L. M. S. {\bf 82} (2001), 229-256. 
 
\bibitem{basterra} M.~Basterra, {\em Andr\'e--Qullen cohomology of commutative $S$--algebras}, J. Pure Appl. Alg. {\bf 144} (1999), 111-143.
    
\bibitem{basterra mccarthy} M.~Basterra and R.~McCarthy, {\em $\Gamma$--homology, topological Andr\'e--Quillen homology, and stabilization}, Top. Appl. {\bf 121} (2002), 551-566.
    
\bibitem{behrens rezk} M.~Behrens and C.~Rezk, {\em The Bousfield-Kuhn functor and topological André-Quillen cohomology},  Invent. Math. {\bf 220} (2020), 949-1022.  
    
\bibitem{ching} M.~Ching, {\em Bar constructions for topological operads and the Goodwillie derivatives of the identity}, Geom. Topol. {\bf 9} (2005), 833-933.
    
\bibitem{fresse} B.~Fresse,  Modules over operads and functors, 
Lecture Notes in Mathematics {\bf 1967}, Springer-Verlag, Berlin, 2009.

\bibitem{harper bar construction} J.E.Harper, {\em Bar constructions and Quillen homology of modules over operads}, Alg. Geom. Top. {\bf 10} (2010), 87-136.

\bibitem{harper and hess}  J.E.Harper and K.Hess, {\em Homotopy completion and topological Quillen homology of structured ring spectra}, Geom. Topol. {\bf 17} (2013), 1325–1416.
    
\bibitem{hss} Mark Hovey, Brooke Shipley, and Jeff Smith, {\em Symmetric spectra}, J. Amer. Math. Soc. {\bf 13} (2000), 149–208.


\bibitem{K tensor R} N.~J.~Kuhn, {\em The McCord model for the tensor product of a space and a commutative ring spectrum}, Categorical Decomposition Techniques in Algebraic Topology,  Proc. Isle of Skye, Scotland, 2001, Birkhauser Verlag Progress in Math {\bf 215} (2003), 213-236.

\bibitem{kuhn hurewicz} N.~J.~Kuhn, {\em Adams filtration and generalized Hurewicz maps for infinite loopspaces}, Invent. Math. {\bf 214} (2018), 957-998.
    
\bibitem{kuhn pereira} N.~J.~Kuhn and L.~A.~Pereira, {\em Operad bimodules, and composition products on Andr\'e--Quillen filtrations of algebras}, Algebraic and Geometric Topology {\bf 17} (2017), 1105-1130.
    
\bibitem{pereira hha 2016}  L.~A.~Pereira, {\em Cofibrancy of operadic constructions in positive symmetric spectra}, Homology Homotopy Appl. {\bf 18} (2016), no. 2, 133-168.

\end{thebibliography}
\end{document}